\documentclass[12pt]{amsart}

\usepackage{amssymb,latexsym}
\usepackage{mathtools}
\usepackage{enumerate}  
\usepackage{a4wide}    

\usepackage[colorlinks=true,citecolor=blue,linkcolor=blue]{hyperref}
\usepackage{latexsym,amsmath,amssymb}
\usepackage{accents}
\usepackage{color}  
\title{H\"older-Topology of the Heisenberg group}

\author{Armin Schikorra}

\address[Armin Schikorra]{Department of Mathematics,
University of Pittsburgh,
301 Thackeray Hall,
Pittsburgh, PA 15260, USA}
\email{armin@pitt.edu}

plus 6pt minus 12pt
\def\eps{\varepsilon}

\def\id{{\rm id\, }}

\def\B{{\mathbb B}}

\def\M{{\mathcal M}}
\def\N{{\mathbb N}}
\def\H{{\mathbb H}}

\def\S{{\mathbb S}}

\newtheorem{theorem}{Theorem}
\newtheorem{lemma}[theorem]{Lemma}

\newtheorem{proposition}[theorem]{Proposition}

\theoremstyle{definition}
\newtheorem{remark}[theorem]{Remark}
\newtheorem{definition}[theorem]{Definition}

\newtheorem{conjecture}[theorem]{Conjecture}

\def\lip{{\rm Lip\,}}
\def\rank{{\rm rank\,}}
\def\supp{{\rm supp\,}}

\newcommand{\R}{\mathbb{R}}

\newcommand{\HI}{\mathcal{H}}
\newcommand{\Z}{\mathbb{Z}}

\newcommand{\brac}[1]{\left (#1 \right )}

\newcommand{\Ep}{\bigwedge\nolimits}

\newcommand{\barint}{
\rule[.036in]{.12in}{.009in}\kern-.16in \displaystyle\int }

\newcommand{\barcal}{\mbox{$ \rule[.036in]{.11in}{.007in}\kern-.128in\int $}}

\def\mvint_#1{\mathchoice
          {\mathop{\vrule width 6pt height 3 pt depth -2.5pt
                  \kern -8pt \intop}\nolimits_{\kern -3pt #1}}%
          {\mathop{\vrule width 5pt height 3 pt depth -2.6pt
                  \kern -6pt \intop}\nolimits_{#1}}%
          {\mathop{\vrule width 5pt height 3 pt depth -2.6pt
                  \kern -6pt \intop}\nolimits_{#1}}%
          {\mathop{\vrule width 5pt height 3 pt depth -2.6pt
                  \kern -6pt \intop}\nolimits_{#1}}}

\numberwithin{theorem}{section} \numberwithin{equation}{section}

\newcommand{\lap}{\Delta }

\newcommand{\aleq}{\lesssim}

\begin{document}
\begin{abstract}
The Heisenberg groups are examples of sub-Riemannian manifolds homeomorphic, but not diffeomorphic to the Euclidean space. Their metric is derived from curves which are only allowed to move in so-called horizontal directions.

We report on some recent progress in the Analysis of the H\"older topology of the Heisenberg group, some related and some unrelated to density questions for Sobolev maps into the Heisenberg group.

In particular we describe the main ideas behind a result by Haj\l{}asz, Mirra, and the author regarding Gromov's conjecture, which is based on the linking number. We do not prove or disprove the Gromov Conjecture.
\end{abstract}

\maketitle 
\tableofcontents

\section{\texorpdfstring{$H=W$}{H=W}-problem for maps into manifolds and the role of topology}
When are Sobolev maps into manifolds \emph{essentially} smooth? 

To study this question, let us be more precise. Let $\mathcal{N} \subset \R^N$ be a (for presentations sake) smooth, compact Riemannian manifold without boundary embedded in the Euclidean space $\R^N$. There are two ways to define the Sobolev space of maps from a $d+1$-dimensional domain $\Omega \subset \R^{d+1}$ into $\mathcal{N}$.
\begin{enumerate}
 \item (The restriction space) Let $W^{1,p}(\Omega,\mathcal{N})$ be all maps $u \in W^{1,p}(\Omega,\R^N)$ so that $u(x) \in \mathcal{N}$ for almost every $x \in \Omega$.
 \item (The functional analytic space) Let $C^{\infty}(\Omega,\mathcal{N})$ be all smooth functions maps from $\Omega$ into $\mathcal{N}$. Set 
 \[
  H^{1,p}(\Omega,\mathcal{N}) := \operatorname{closure}_{\|\cdot\|_{W^{1,p}}} (C^{\infty}(\Omega,\mathcal{N})),
 \]
to be the closure of $C^{\infty}(\Omega,\mathcal{N})$ under the $W^{1,p}$-norm
\[
 \|f\|_{W^{1,p}} = \brac{\|f\|_{L^p}^p + \|Df\|_{L^p}^p}^{\frac{1}{p}}.
\]
\end{enumerate}
The above question can then be translated into: when is $H = W$, i.e. $H^{1,p}(\Omega,\mathcal{N}) = W^{1,p}(\Omega,\mathcal{N})$?

(Actually, it is always true the $H^{1,p}(\Omega,\mathcal{N}) \subset W^{1,p}(\Omega,\mathcal{N})$, the other direction is the interesting one).

The answer of this question depends on the topology of $\mathcal{N}$, and also of $\Omega$. More precisely, the homotopy classes of $\mathcal{N}$ play a role. 
\begin{definition}
Two continuous maps $f,g: \S^k \to \mathcal{N}$ belong to the same homotopy class (which we denote as elements of $\pi_k(\mathcal{N})$) if they can be continuously transformed into one another, that is there exists a homotopy $H: [0,1] \times \S^k \to \mathcal{N}$ which is continuous and $H(0,\cdot) = f(\cdot)$ and $H(1,\cdot) = g(\cdot)$. It is easy to check that this induces an equivalence class.

We say that $\pi_k(\mathcal{N})$ is trivial, $\pi_k(\mathcal{N}) = \{0\}$ if and only if for any continuous (or smooth) map $f: \S^k \to \mathcal{N}$ there exists an extended map $F: \overline{\B^{k+1}} \to \mathcal{N}$ such that
\begin{enumerate}
 \item $F$ is continous/smooth on $\overline{\B^{k+1}}$
 \item $F \big |_{\partial \B^{k+1}} = f$.
\end{enumerate}
That is to say, any map is (continously/smoothly) deformable to a constant map.
\end{definition}

The following is the celebrated result by Bethuel \cite{Bethuel-1991}, see also
\cite{Bethuel-Zheng-1988,Hang-Lin-2003}.

\begin{theorem}[\cite{Bethuel-1991}]\label{th:HWBethuel}
Let $\Omega = \B^{d+1} \subset \R^{d+1}$ be a ball
\begin{itemize}
 \item Assume $p \geq d+1$. Then $W^{1,p}(\B^{d+1},\mathcal{N}) = H^{1,p}(\B^{d+1},\mathcal{N})$.
 \item Assume $1 \leq p < d+1$. Then $W^{1,p}(\B^{d+1},\mathcal{N}) = H^{1,p}(\B^{d+1},\mathcal{N})$ if and only if $\pi_{\lfloor p\rfloor} \mathcal{N} = \{0\}$.
\end{itemize}
\end{theorem}
The first statement is actually due to Schoen-Uhlenbeck \cite{Schoen-Uhlenbeck-1983}. Many extensions, e.g. to more general Sobolev spaces, exists; see \cite{Hajlasz-1994,Riviere-2000,Mironescu-2004,Bousquet-Ponce-VanSchaftingen-2013,Brezis-Mironescu-2015,Bousquet-Ponce-VanSchaftingen-2015} and references within.

\emph{Why is the topology of the target manifold so important? How does a nontrivial homotopy group obstruct smooth approximation?}

Let us look at a special case:
\begin{proof}[Sketch of the ``only if'' argument for $p \in (d,d+1)$.]
Assume that $\pi_{\lfloor p\rfloor} \mathcal{N} \neq \{0\}$. Then we can find a (w.lo.g. smooth) map $\varphi : \S^{d} \to \mathcal{N}$ that is nontrivial as an element of $\pi_{d}(\mathcal{N})$. That is, \emph{any} map $\Phi: \B^{d+1} \to \mathcal{N}$ which coincides with $\varphi$ at the boundary $\Phi \big |_{\partial \B^{d+1}} = \varphi$ is necessarily \emph{discontinuous}.
So we cannot extend $\varphi$ to a \emph{continuous} map on $\B^{d+1}$. On the other hand, we can easily extend $\varphi$ to a map on $\B^{d+1}$ which is in $W^{1,p}(\B^{d+1},\mathcal{N})$. Simply take 
\[
 \Phi(x) := \varphi\left (\frac{x}{|x|} \right ) \in W^{1,q}(\B^{d+1},\mathcal{N})\quad  \forall 1 \leq q < d+1.
\]
This map $\Phi: \B^{d+1} \to \mathcal{N}$ cannot be approximated in $W^{1,p}$ by smooth functions $\Phi_k \in W^{1,p}(\B^{d+1},\mathcal{N})$. This is because of the topological restriction that $\Phi$ bridges, but continuous functions $\Phi_k$ cannot bridge. More precisely, assume that $\Phi_k \to \Phi$ in $W^{1,p}(\B^{d+1},\R^N)$. Essentially by Fubini's theorem we can find a radius $r \in (0,1)$ so that as functions restricted to the $r$-sphere $r \S^{d}$ we have convergence
\[
 \Phi_k\big |_{r \S^{d}}\xrightarrow{k \to \infty} \Phi\big |_{r \S^{d}} \quad \mbox{in $W^{1,p}(r\S^{d},\mathcal{N}$)}.
\]
In that case we have three facts:
\begin{enumerate}
 \item As an element of $\pi_k(\mathcal{N})$ the map $\Phi_k\big |_{r \S^{d}}: r \S^{d} \to \mathcal{N}$ is trivial: it can be \emph{continuously} extended to all of $\B^{d+1}$ by $\Phi_k$.
 \item As an element of $\pi_k(\mathcal{N})$ the map $\Phi\big |_{r \S^{d}}: r \S^{d} \to \mathcal{N}$ is nontrivial: $\Phi(rx) = \varphi(x)$ on $\S^{d}$.
 \item $\Phi_k\big |_{r \S^{d}} \xrightarrow{k \to \infty} \Phi\big |_{r \S^{d}}$ uniformly: Since $p > d$ and the latter is the dimension of the $r\S^{d}$, we have that $W^{1,p}(r\S^{d}) \subset C^0(r\S^{d})$.
\end{enumerate}
The third fact implies that for $k \in \N$ sufficiently large, $\Phi_k\big |_{r \S^{d}}$ and $\Phi\big |_{r \S^{d}}$ are the same as elements of $\pi_{d-1}(\mathcal{N})$, since they are uniformly close to each other. This makes (1) and (2) impossible, and we have a contradiction.
\end{proof}
A crucial ingredient in the argument above, albeit somewhat hidden, is that we can jump as we wish between smooth and continuous elements of the homotopy group $\pi_{d-1}(\mathcal{N})$. Behind this lies the following density fact, which follows directly from convolution arguments and the existence of a nearest-point projection $\Pi: B_\delta(\mathcal{N}) \to \mathcal{N}$ in a small tubular neighborhood of $B_\delta(\mathcal{N})$ for some $\delta > 0$.
\begin{lemma}\label{la:continuouslipschitz}
Let $u: \Omega \to \mathcal{N}$ be continuous ($C^\alpha$-H\"older continuous, $C^{0,1}$-Lipschitz continuous). Then we can approximate $u$ by smooth or Lipschitz-continuous $u_k$ so that $u_k$ converges to $u$ in $C^0$, ($C^\beta$ for $\beta < \alpha$,or $1$, respectively).
\end{lemma}

This fails if $\mathcal{N}$ becomes a metric space $X$. One example of a metric space with still a lot of ``smooth structure'' are the Heisenberg groups.

\section{A crude introduction to the Heisenberg group}
The metric of a manifold $X \subset \R^{N}$ could be described in the following way. To measure the distance between any two points $p, q \in X$, we take the minimum length of \emph{tangential} curves between $p$ and $q$ in $X$: a smooth curve $\gamma:[0,1] \to \R^N$ is a \emph{tangential} curve between $p$ and $q$ if
\begin{enumerate}
 \item $\gamma(0) = p$, $\gamma(1) = q$, and
 \item at any $t \in (0,1)$, the derivative of $\gamma$ belongs to the tangential space $T_{\gamma(t)} X$ of $X$.
\end{enumerate}
Of course, for $X$ a smooth embedded manifold this is just equivalent to saying that $\gamma(t)$ maps into the manifold $X$ at any point $t \in [0,1]$. So in some sense it is equivalent to see a manifold $X$ as a distribution of tangent planes $T_p X$. Why are the two points of view equivalent? It is essentially Frobenius theorem, the tangent plane distribution is integrable (see, e.g., \cite{Lang-1999}).

Now we drop the integrability condition of the (previously tangent) planes, and call them \emph{horizontal planes} $H_p X$: For any $p \in X$ define a linear space $H_p X \subset \R^N$. Let us also assume that the spaces $H_p X$ vary smoothly (in a suitable sense). We define a metric $d_X(p,q)$ as the minimal length of horizontal curves connecting $p$ and $q$. Horizontal curve connecting $p$ and $q$ simply means 
\begin{enumerate}
 \item $\gamma(0) = p$, $\gamma(1) = q$, and
 \item at any $t \in (0,1)$, the derivative of $\gamma$ belongs to the tangential space $T_{\gamma(t)} X$ of $X$.
\end{enumerate}
This metric is called the Carnot-Carath\'eodory metric, and $(X,HX,\R^N)$ is a \emph{sub-Riemannian}-manifold. Observe that in principle the distance between two points could be infinite.

One of the simplest (non-trivial) examples of a sub-Riemannian manifold is the Heisenberg group $\H_n$. We take $X = \R^{2n+1}$, and define 
\[\mathcal{L}(\gamma) = \int_0^1 |\gamma'(t)| \quad \mbox{for any absolutely continuous curve $\gamma: [0,1] \to X$.}
\]
At a point $p = (p_1,\ldots,p_{2n+1}) \in \R^{2n+1}$, the horizontal plane distribution is given by the kernel of a $1$-form $\alpha$, $H_p \H_n := \ker \alpha$, where
\begin{equation}\label{eq:alpha}
 \alpha := dp_{2n+1} + 2 \sum_{j=1}^n 
  (p_{2j}\, dp_{2j-1}- p_{2j-1}\, dp_{2j}).
\end{equation}
That is,
\[
 H_p \H_n = \left \{v = (v_1,\ldots,v_{2n+1}) \in \R^{2n+1}:\ v_{2n+1} + 2 \sum_{j=1}^n 
  (p_{2j}\,v_{2j-1}- p_{2j-1}\, v_{2j}) = 0 \right \}.
\]

One can show that under the resulting Carnot-Carath\'eodory metric $d_{cc}$ the metric space $(\H_n,d_{cc})$ is connected. Also, $d_{cc}$ is equivalent to the so-called Kor\'anyi-metric $d_{\H_n}$ which we shall use from now on. It is defined as follows: for $p = (p_1,\ldots,p_{2n+1})$, $q = (q_1,\ldots,q_{2n+1})$ by
\begin{equation}\label{eq:dcc}
 d_{\H_n}(p,q) = \left (\sum_{i=1}^{2n} \left |p_{i}-q_i \right |^4+ \Big |p_{2n+1} -q_{2n+1} + 2 \sum_{j=1}^n \det \left ( \begin{array}{cc} 
 p_{2j-1}-q_{2j-1}& q_{2j-1}\\
 p_{2j}-q_{2j}& q_{2j}
 \end{array} \right ) \Big|^2 \right )^{\frac{1}{4}}.
\end{equation}

A remark in passing: The Heisenberg group is called a group since it really has a Lie group structure. For two elements $(z,t)$ and $(z',t') \in \mathbb{C}^n \times \R \equiv \R^{2n+1}$ the group law of the Heisenberg group is
$$
(z,t)*(z',t')=\left(z+z',t+t'+2\, {\rm Im}\,  \left(\sum_{j=1}^n z_j
  \overline{z_j'}\right)\right).
$$
If we write $z = x+iy$, a basis of left invariant vector fields is given by
\begin{equation}
\label{XY}
X_j=\frac{\partial}{\partial x_j} + 2y_j\frac{\partial}{\partial t},\
Y_j=\frac{\partial}{\partial y_j}-2x_j\frac{\partial}{\partial t},\ j=1,\ldots,n, \
\mbox{and}\
T=\frac{\partial}{\partial t}\, .
\end{equation}
The horizontal space $H_{(z,t)}\H_n$ is then spanned by $X_1,\ldots,X_n,Y_1,\ldots,Y_n$. 

In particular a geodesic (or in fact any Lipschitz curve) cannot move ``straight up'' in the $(2n+1)$th $T$-direction. When you want to find a geodesic between the origin $(0,\ldots,0,0)$ with $(0,\ldots,0,t)$ one needs to circle around the $2n+1$th axis. For the structure of geodesics we refer to \cite{Hajlasz-Zimmerman-2015}.

We will not go more into details of the geometry of the Heisenberg group. The interested reader is referred to \cite{Capogna-Danielli-Pauls-Tyson-2007}.

\subsection{Horizontal maps and Sobolev maps into the Heisenberg group}
An easy computation implies that any map into the Heisenberg group which is Lipschitz with respect to the metric $d_{\H_n}$ satisfies almost everywhere $f^\ast \alpha = 0$.
\begin{lemma}\label{la:lipschitzhorizontal}
Let $f: \Omega \subset \R^d \to \H_n$ be Lipschitz, that is
\[
 d_{\H_n}(f(x),f(y)) \aleq\ |x-y|. 
\]
Then, at any point where $f$ is differentiable, for $\alpha$ the 1-form from \eqref{eq:alpha}
\[
f^\ast \alpha = 0. 
\]
\end{lemma}
\begin{proof}
The Lipschitz-condition implies in particular
\[
 \left |\frac{f_{2n+1}(x) -f_{2n+1}(x+h)}{|h|} + 2 \sum_{j=1}^n \det \left ( \begin{array}{cc} 
 \frac{f_{2j-1}(x)-f_{2j-1}(x+h)}{|h|}& f_{2j-1}(x+h)\\
 \frac{f_{2j}(x)-f_{2j}(x+h)}{|h|}& f_{2j}(x+h)
 \end{array} \right ) \right |  \aleq |h|.
\]
It is worth noting, and becomes important later, that this computation actualy shows in some sense: if $f \in C^{\frac{1}{2}+\eps}(\Omega,\H_n)$ then $f^\ast(\alpha) = 0$. Cf. Proposition~\ref{pr:approx}.
\end{proof}

Maps $f: \Omega \to \H_n$ which satisfy $f^\ast\alpha = 0$ are called \emph{horizontal maps}. The Sobolev space $W^{1,p}(\Omega,\H_n)$ is defined as 
\[
 W^{1,p}(\Omega,\H_n) = \left \{ f \in W^{1,p}(\Omega,\R^{2n+1}) \quad \mbox{such that }\ f^\ast \alpha = 0 \mbox{ a.e. in $\Omega$}\right \}.
\]
Another definition of $W^{1,p}(\Omega,\H_n)$ is as follows: embed the separable metric space $(\H_n,d_{\H_n})$ in $\ell^\infty$ with the Kuratowski-embedding. Sobolev maps $W^{1,p}$ into the Banach space $\ell^\infty$ are well-defined, and $W^{1,p}(\Omega,\H_n)$ are all those maps $f \in W^{1,p}(\Omega,\ell^\infty)$ that pointwise a.e. belong to $\H_n\subset \ell^\infty$. Both definitions coincide, \cite{Capogna-Lin-2001,Dejarnette-Hajlasz-Lukyanenko-Tyson-2014}.

\subsection{Topology of the Heisenberg group}
When we want to study density questions in Sobolev spaces $W^{1,p}$ (since $\H_n$ is a metric space, we will ask for density of Lipschitz mappings), in view of Bethuel's theorem, Theorem~\ref{th:HWBethuel}, we would like to understand the topology of $\H_n$. But we stumble over the following (simple) fact:
\begin{proposition}
All homotopy classes of $\H_n$ are trivial, i.e. $\pi_k(\H_n) = 0$ for any $k \in \N$.
\end{proposition}

This is actually quite easy to see: every map into $\R^{2n+1}$ which is continuous with respect to the Euclidean metric is also continuous with respect to the $\H_n$-metric. We even have

\begin{lemma}\label{la:comparison}
We can estimate this distance in terms of the usual $\R^{2n+1}$-norm $|\cdot|$,
\[
 \frac{1}{(|p|+|q| + 1)}\, |p-q| \aleq d_{\H}(p,q) \aleq\ (|p|^{\frac{1}{2}}+|q|^{\frac{1}{2}} + 1)\, |p-q|^{\frac{1}{2}}.
\]
In particular, $(\H_n,d_{\H_n})$ is homeomorphic to $(\R^{2n+1},|\cdot|)$.
\end{lemma}
In particular, any map $f: \Omega \subset \R^d \to \R^{2n+1}$ that is Lipschitz with respect to the $\H_n$-metric is also (locally) Lipschitz with respect to the Euclidean metric. But the converse is false, maps which are Lipschitz with respect to the Euclidean metric might be merely $C^{\frac{1}{2}}$ with respect to the $\H_n$-metric. For example, considered as a map into the Heisenberg group even the identity $\id: \R^{2n+1} \to \R^{2n+1}$ is only $C^{\frac{1}{2}}_{loc}(\R^{2n+1},\H_n)$. 

Actually, one can show that the Hausdorff dimension of any open set in $\H_n$ equals $2n+2$. Thus we have the disturbing situation that $(\H_n,d_{\H_n})$ is homeomorphic to $\R^{2n+1}$ (actually even $C^{\frac{1}{2}}$-homeomorphic), but not Lipschitz-homeomorphic, not even locally. 

\section{H\"older-Topology and density results on the Heisenberg group}
So in order to understand questions of density in Sobolev spaces, we need to find suitable nontrivial topological quantities on the Heisenberg group.
\begin{definition}[$C^\gamma$-homotopy]
For $\gamma \in (0,1]$ the $k$-th $C^\gamma$-homotopy group $\pi_k^\gamma(X)$ of a metric space $X$ is defined as the class of maps $f \in C^\gamma(\S^k,\H_n)$, where two maps $f, g \in C^\gamma(\S^{k},\H_n)$ are identified, if there exists a $C^\gamma$ homotopy $H \in C^\gamma([0,1] \times \S^{k},\H_n)$ so that $H(0,\cdot) = f$ and $H(1,\cdot) = g$.

We write $\pi^{\lip}_k(X)$ for $\pi^{1}_k(X)$.
\end{definition}
Here is what is known on homotopy groups:
\begin{theorem}[Homotopy groups]${}$
\begin{enumerate}
 \item $\pi^{\lip}_m(\H_n)= \{0\}$ for all $1\leq m<n$.
\item $\pi_m^{\lip}(\H_1)=\{ 0\}$ for all $m\geq 2$.
 \item $\pi^{\rm \gamma}_n(\H_n)\neq \{0\}$ when $\frac{n+1}{n+2} < \gamma \leq 1$.
 \item $\pi_{4n-1}^\gamma(\H_{2n}) \neq \{0\}$ when $\frac{4n+1}{4n+2} < \gamma \leq 1$.
\end{enumerate}
 \end{theorem}
(1) was proven in \cite{Dejarnette-Hajlasz-Lukyanenko-Tyson-2014,WengerY1}. (2) is due to \cite{WengerY2}. For (3) there are several proofs in the Lipschitz case \cite{Balogh-Faessler-2009,Dejarnette-Hajlasz-Lukyanenko-Tyson-2014,Hajlasz-Schikorra-Tyson-2014}. For Lipschitz homotopy groups, (4) was the main result of \cite{Hajlasz-Schikorra-Tyson-2014}. The H\"older-groups are from a forthcoming paper \cite{Hajlasz-Mirra-Schikorra-2016}. In a more recent paper, \cite{H18} Haj\l{}asz proved that for $n \geq 2$, $\pi^{n+1}_{\lip}(\H_{n}) \neq \{0\}$.

It may seem natural to hope that the counterpart of Theorem~\ref{th:HWBethuel} holds for Lipschitz Homotopy $\pi_k^{\lip}(\H_n)$. For example, $\pi_d^\lip(\H_n) = 0$ if and only if Lipschitz maps are dense in $W^{1,p}(\B^{d+1},\H_n)$ for $d < p < d+1$. However, we do not know this: we cannot just run the algorithm for Bethuel's Theorem~\ref{th:HWBethuel} described above. $W^{1,p}(\S^d)$ on a $d$-dimensional manifold embeds merely into $C^{1-\frac{d}{p}}$, and $1-\frac{d}{p} < \frac{1}{d+1} \leq \frac{1}{2}$, so convergence in $W^{1,p}(\S^d,\H_n)$ means nothing in terms of convergence in Homotopy groups.

The counterpart of Lemma~\ref{la:continuouslipschitz} is unknown in the Heisenberg group (and it actually false for $C^{\gamma}$, $\gamma < \frac{1}{2}$). Related to this we do not know (although it seems quite likely at least for $\gamma \approx 1$ whether $\pi_k^\gamma (\H_n) = \pi_k^\lip(\H_n)$ for any $k$. The technical issue with this is that in contrast with Riemannian manifolds here we do not have a projection $\Pi$ that could map non-horizontal lines which are uniformly close to a horizontal line into the ``nearest horizontal line'', so we cannot (i.e. do not understand how to) approximate even H\"older maps with Lipschitz maps.

Even though it is not an immediate consequence of the non-triviality of the corresponding Lipschitz-Homotopy groups, we the following non-density results are known:
\begin{theorem} Let $\mathcal{M}$ be a smooth compact Riemannian manifold possibly with boundary.
\begin{enumerate}
\item If $\dim\M \leq n$ then the Lipschitz maps $\lip(\M,\H_n)$ are dense in $W^{1,p}(\M,\H_n)$, for any $1 < p <\infty$.
\item If $\dim\M\geq n+1$ and $n\leq p<n+1$, then
Lipschitz maps $\lip(\M,\H_n)$ are not dense in
$W^{1,p}(\M,\H_n)$. 
\item If $\M$ is a compact Riemannian manifold with or without boundary of
dimension $\dim\M\geq 4n$, then Lipschitz mappings
$\lip(\M,\H_{2n})$ are not dense in $W^{1,p}(\M,\H_{2n})$ when
$4n-1\leq p<4n$.
\end{enumerate}
\end{theorem}
(1) is due to \cite{Dejarnette-Hajlasz-Lukyanenko-Tyson-2014}, see also \cite{Hajlasz-Schikorra-2014}.
(2) is due to \cite{Dejarnette-Hajlasz-Lukyanenko-Tyson-2014}, (3) is from \cite{Hajlasz-Schikorra-Tyson-2014}.

Let us remark another interesting topological fact of the Heisenberg group, even though it is not (necessarily) related to density questions.
\begin{theorem}\label{th:gromov}
Let $k \geq n+1$, $\gamma \in (\frac{1}{2},1]$, $\theta > 0$ and
\begin{equation}\label{eq:gammatheta}
 2\gamma + \theta (k-1) -k > 0.
\end{equation}
Then there is no injective $f: \Omega \subset \R^k \to \H_n$ which is $C^\gamma$ with respect to the $\H_n$-metric and $C^\theta$ with respect to the Euclidean metric.
\end{theorem}
Theorem~\ref{th:gromov} is from the forthcoming paper \cite{Hajlasz-Mirra-Schikorra-2016}. For $\gamma = \theta > \frac{k}{k+1}$ it was proven by Gromov, \cite{GromovCarnotCaratheodory}, see also Pansu's \cite{Pansu-2016},  using microflexibility arguments. We also refer to the work by LeDonne and Z\"ust \cite{LeDonne-Zust-2012}. For $\theta = 1$ this can be found in \cite{BaloghHajlaszWildrick}. Theorem~\ref{th:gromov} is so far among the closest results we have to proving a conjecture by Gromov, see also the recent 
\cite{WengerYoung18} and references within.

\begin{conjecture}[Gromov]
There is no embedding $f \in C^{\gamma}(\Omega,\H_n)$ whenever $\Omega$ is an open subset of $\R^k$, $k \geq n+1$, and $\gamma > \frac{1}{2}$.
\end{conjecture}

Let us remark that there is a construction due to Haj\l{}asz and Mirra \cite{H14} that might serve as a counterexample to the Gromov's conjecture, or at least show that there are embedded curves into the Heisenberg group $\H_1$ that can be extended to a $C^{2/3}$-map (not necessarily embedded). Currently the H\"older regularity of this construction can only be measured from a numerical point of view, but this numerical evidence hints toward a $C^{2/3}$-regularity rather than a $C^{1/2}$-regularity as predicted by Gromov's conjecture. The details will be published in the forthcoming \cite{Hajlasz-Mirra-Schikorra-2016}. Also the recent \cite{WengerYoung18} can be interpreted towards that direction.

\section{Ingredient: rank-condition for Lipschitz-maps}
In this section we state the main reason that -- while working on the Heisenberg-group -- we actually don't need to work with the Heisenberg group: derivatives of Lipschitz maps into the Heisenberg group (below we will see what to do with H\"older maps) have a low rank.

We recall the so-called contact-form $\alpha$ whose kernel is the horizontal space distribution of the Heisenberg group $\H_n$,
\[\tag{\ref{eq:alpha}}
 \alpha := dp_{2n+1} + 2 \sum_{j=1}^n 
  (p_{2j}\, dp_{2j-1}- p_{2j-1}\, dp_{2j}).
\]
Note that
\[
 d\alpha = 4 \sum_{j=1}^n   dp_{2j} \wedge dp_{2j-1}.
\]

From Lemma~\ref{la:lipschitzhorizontal} we learned that any map $f \in \lip(\Omega,\H_n)$ satisfies
\[
f^\ast (\alpha) = 0. 
\]
Clearly, this implies also
\[
f^\ast (d\alpha) = 0. 
\]
It is a not difficult but a lengthy, combinatorial proof to show that any $(n+1)$-form can be decomposed into terms containing $\alpha$ or $d\alpha$. 

Actually the following is well-known to experts as the a version of the Lefschetz-Lemma.

\begin{lemma}\label{la:rank}
For any $k \geq n+1$, any $k$-form $\kappa$ has the form
\[
\kappa = \alpha \wedge \beta + d\alpha \wedge \sigma
\]
for some $(k-1)$-form $\beta$ and some $(k-2)$-form $\sigma$.

In particular, if $f \in \lip(\Omega,\H_n)$, then
\[
 f^\ast (\kappa) = 0\quad \mbox{for any $k$-form $\kappa$}.
\]
\end{lemma}
\begin{proof}
We only discuss the three-dimensional situation. For the general $2n+1$ more combinatorical reasoning is needed. 

Take $(x,y,z) \in \R^3$ and a $2$-form $\kappa$
\[
 \kappa = \kappa_1\, dy \wedge dz + \kappa_2\, dx \wedge dz +  \kappa_3\, dx \wedge dy.
\]
Observe $d\alpha = 4 dx \wedge dy$, and $\alpha = dz + 2  (y\, dx- x\, dy)$. Thus
\[
 \kappa_1\, dy \wedge dz = \kappa_1\, dy \wedge \alpha -2 \kappa_1 \,y\, dy \wedge dx = \kappa_1\, dy \wedge \alpha -\frac{1}{2} \kappa_1 \,y\, d\alpha,
\]
\[
 \kappa_2\, dx \wedge dz = \kappa_2\, dx \wedge \alpha + 2 \kappa_2\, y\, dx \wedge dy = \kappa_2\, dx \wedge \alpha - \frac{1}{2} \kappa_2\, y\, d\alpha,
\]
and
\[
 \kappa_3\, dx \wedge dy = -\frac{1}{4}\kappa_3\, d\alpha.
\]

\end{proof}

An equivalent formulation for Lipschitz functions (but as we shall see, the above statement is more useful for H\"older functions)
\begin{lemma}
Let $f \in \lip(\Omega,\H_n)$, then
\[
 \rank Df \leq n \quad \mbox{a.e. in $\Omega$}.
\]
\end{lemma}
This is a very rigid statement, recall that the $n$-th Heisenberg group $\H_n$ is homeomorphic to $\R^{2n+1}$!

\section{Ingredient: linking number}
\begin{proposition}\label{pr:linking}
Let $k < N-1$ and $\varphi: \S^k \to \R^{N}$ a Lipschitz embedding. Then there exists a smooth $k$-form $\omega$ on $\R^N$ so that
\[
 \int_{\S^k} \varphi^\ast (\omega) \neq 0.
\]
\end{proposition}
The reason for this to be true is the linking number. Usually the linking number $\mathcal{L}(A,B)$ of a $k$-dimensional (closed) object $A$ and disjoint a (closed) $N-k-1$-dimensional object measures how many times object $A$ winds around object $B$. 'closed' means that $A$ and $B$ have no boundaries (and are in fact a boundary of a $k+1$ and a $N-k$-dimensional object, respectively). For $N=3$, $k=1$ both objects are just curves.

Any $(N-k-1)$-dimensional closed object $B$ can be measured by an closed (and thus exact) $k+1$ differential form $\eta_B = d\omega_B$. This is Poincar\'{e}-duality. The disjointness of $A$ and $B$ is just that $\eta_B$ has no support in $A$.

In algebraic terms, the linking number $\mathcal{L}(A,B)$ is the homology class of $B$ in $H_{N-k-1}(\R^N \backslash A,\Z)$ or equivalently the cohomology class of $\eta_B$ in $H^{k+1}(\R^N \backslash A,\Z)$.

In analytic terms, the linking number is
\[
 \mathcal{L}(A,B) \equiv \mathcal{L}(A,\eta_B) = \int_{A} \omega_B,
\]
which simply means, as is shown in \cite{Hajlasz-Mirra-Schikorra-2016}, that the map
\[
 \eta = d\omega \mapsto \int_{\S^k} \varphi^\ast (\omega) 
\]
is an isomorphism on $H_{k+1}(\R^N \backslash A,\Z)$.

So the statement of Proposition~\ref{pr:linking} is simply saying that if $\varphi(\S^k)$ is an embedded $k$-sphere in $\R^N$, then there exists some object $B$ linked to it. The latter is a standard fact from algebraic topology, and we adapt the standard proof, see e.g. \cite[Corollary 1.29]{Vick-1994}. We will sketch the proof in Section~\ref{s:linkingproof}.

\subsection{Implication for Gromov's theorem (Lipschitz case)}\label{s:lipschitzgromov}
Observe that Proposition~\ref{pr:linking} implies in particular the Lipschitz version of Gromov's result, Theorem~\ref{th:gromov}. 

Let $\Phi: \B^{n+1} \to \H_n$ be a Lipschitz embedding. In particular, $\Phi$ is Lipschitz as a map into $\R^{2n+1}$. Let $\varphi := \Phi \big |_{\S^{n}}$ be the boundary map of $\Phi$, which is of course still an embedding. In view of Proposition~\ref{pr:linking} we find a $n$-form $\omega$ in $\R^{2n+1}$ so that
\[
 0 \neq \int_{\S^n} \varphi^\ast (\omega).
\]
With Stokes' theorem
\[
 =\int_{\B^{n+1}} \Phi^\ast (d\omega).
\]
Since $d\omega$ is an $n+1$-form, and the rank-condition, Lemma~\ref{la:rank}, tells us that
\[
 = 0.
\]
We have a contradiction, so $\Phi$ could not have been an embedding. 

Actually we even showed 
\begin{lemma}\label{la:noextension}
No Lipschitz embedding $\varphi: \S^n \to \H_n$ can be Lipschitz extended to $\Phi: \B^{n+1} \to \H_n$.
\end{lemma}

\section{Lipschitz case: \texorpdfstring{$\pi_n^{\lip}(\H_n)$}{}}
\subsection{Non-triviality}
Theorem~\ref{th:gromov} tells us that it is impossible to Lipschitz-embeds objects in to the $\H_n$ if their dimension is larger than $n+1$. This bound on the dimension is sharp, the following was shown by \cite[Section~4]{Balogh-Faessler-2009}, \cite[Theorem~3.2]{Dejarnette-Hajlasz-Lukyanenko-Tyson-2014}, \cite[Example~3.1]{Ekholm-Etnyre-Sullivan-2005}.
\begin{theorem}\label{th:Snembedding}
For any $n\geq 1$ there is a bi-Lipschitz embedding $\varphi: \S^n \to \H_n$.
\end{theorem}

Clearly, we can consider $\varphi$ to be an element of $\pi_n^{\lip}(\H_n)$. In view of Lemma~\ref{la:noextension} it is a non-trivial element of $\pi_n^{\lip}(\H_n)$.

\subsection{So what about density?} Note that we have a quantitative way to measure the nontriviality of the homotopy group.
Take $\varphi$ from above. As a map into $\R^{2n+1}$, $\varphi$ is an embedding, so in view of Proposition~\ref{pr:linking} we can find a $n$-form $\omega$ so that
\begin{equation}\label{eq:nontrivialityn}
 \int_{\S^n} \varphi^\ast (\omega) \neq 0.
\end{equation}
Now we let the standard algorithm run and obtain non-density of Lipschitz maps for the Sobolev maps $W^{1,p}(\B^{n+1},\H_n)$, $n < p < n+1$.

Take $\varphi$ and $\omega$ from above so that \eqref{eq:nontrivialityn} holds. Set $\Phi(x) := \Phi(x/|x|) \in W^{1,p}(\B^{n+1},\H_n)$ for any $p < n+1$. This $\Phi$ can not be $W^{1,p}$-approximated by Lipschitz maps in $\lip(\B^{n+1},\H_n)$. If there was an approximation $\Phi_k \to \Phi$ in $W^{1,p}(\B^{n+1},\H_n)$, then on some sphere $r\S^{n}$, $r \in (0,1)$ (we pretend for simplicity that $r=1$)
\[
 \Phi_k\big |_{\S^{n}} \xrightarrow{k\to \infty} \varphi \quad \mbox{in $W^{1,p}(\S^n,\H_n)$}.
\]
Since $\Phi_k$ is a Lipschitz map into $\H_n$ and $d\omega$ is a $n+1$-form, by the rank condition, Lemma~\ref{la:rank}, and Stokes' theorem
\[
 0 = \int_{\B^{n+1}} \Phi_k^\ast (d\omega) = \int_{\S^{n}} \Phi_k^\ast (\omega)
\]
Now $\omega$ is an $n$-form, and thus $|\Phi_k^\ast (\omega)| \leq |D\Phi_k|^n\ w(\Phi_k)$ (for some smooth $w$). Since $\Phi_k \to \varphi \in W^{1,p}(\S^n,\H_n)$, $p > n$, the integral above converges. We thus have
\[
 0 = \lim_{k \to \infty} \int_{\S^{n}} \Phi_k^\ast (\omega) = \int_{\S^{n}} \varphi^\ast (\omega)  \overset{\eqref{eq:nontrivialityn}}{\neq} 0,
\]
a contradiction. We conclude that there is no Lipschitz approximation for $\Phi$, and thus Lipschitz functions are not dense in $W^{1,p}(\B^{n+1},\H_n)$ if $p \in (n,n+1)$.

\section{Lipschitz case: \texorpdfstring{$\pi_{4n-1}^{\lip}(\H_{2n})$}{}}
For this we employ another version of linking number, the one that Hopf \cite{Hopf} used to define his Hopf invariant, and showed that $\pi_{4n-1}(\S^{2n}) \neq 0$.
\subsection{Another linking number: the Hopf invariant}
Let $\varphi: \S^{4n-1} \to \S^{2n}$. Take the volume form $\eta$ of $\S^{2n}$. Then $\varphi^\ast(\eta)$ is a closed form: $d\varphi^\ast(\eta)=\varphi^\ast(d\eta) = 0$, since $d\eta$ is a $2n+1$-form; but $\varphi$ is a map into $\S^{2n}$ so surely its derivative $D\varphi$ can only have rank at most $\leq 2n$. But in $\S^{4n-1}$ any closed $2n$-form is exact, so $\varphi^\ast(\eta) = d\omega_\varphi$. The Hopf invariant is then defined as
\[
 \mathcal{H}(\varphi) = \int_{\S^{4n-1}} \omega_\varphi \wedge \varphi^\ast(\eta).
\]
As explained in \cite{BT82}, $\mathcal{H}(\varphi)$ measures the linking number between the two $(2n-1)$-dimensional ``curves'' $\varphi^{-1}(q)$, $\varphi^{-1}(p)$.

Hopf then showed
\begin{theorem}[Hopf \cite{Hopf}]\label{th:hopf}
\label{la:hopffibration}
For any $n \in \mathbb{N}$ there exists a smooth map
$\varphi: \S^{4n-1} \to \S^{2n}$, such that $\HI (\varphi) \neq 0$.
\end{theorem}

\subsection{Adaption to the Heisenberg group}
The main observation is that what makes the Hopf invariant actually homotopy invariant is the rank-condition $\rank D\varphi \leq 2n$. 

So take Theorem~\ref{th:hopf} the nontrivial map $\varphi_1 : \S^{4n-1} \to \S^{2n}$ and from Theorem~\ref{th:Snembedding} the bi-Lipschitz embedding $\varphi_2 : \S^{2n} \to \H_{2n}$. We can Lipschitz extend its inverse $\varphi_2^{-1}: \R^{4n+1} \to \R^{2n+1}$. Set $\varphi := \varphi_2 \circ \varphi_1 $ $\in \lip(\S^{4n-1} ,\H_{2n})$. Again, this is an element of $\pi^{\lip}_{4n-1}(\H_{2n})$ and we will show that it is non-trivial.

Assume on the contrary that $\varphi$ is a trivial element of $\pi^{\lip}_{4n-1}(\H_{2n})$. Then we find a Lipschitz extension of $\varphi$, $\Phi \in \lip(\B^{4n},\H_n)$. Set $\Phi_1 := \varphi_2^{-1} \circ \Phi  \in \lip(\B^{4n},\R^{2n+1})$. 

Take $\eta$ the volume form of $\S^{2n}$, so that
\[
 0 \neq  \HI (\varphi_1) = \int_{\S^{4n-1}} \omega_{\varphi_1} \wedge \varphi_1^\ast(\eta)
\]
Now $\Phi_1^\ast(d\eta) = \Phi^\ast((\varphi_2^{-1})^\ast(d\eta)) = 0$, since $(\varphi_2^{-1})^\ast(d\eta)$ is an $(2n+1)$-form, and we have again the rank-condition Lemma~\ref{la:rank}. Thus we find $\omega_{\Phi_1}$ so that $\Phi_1^\ast(\eta) = d\omega_{\Phi_1}$.
We then use Stokes' theorem,
\[
 = \int_{\B^{4n}} d \brac{\omega_{\Phi_1} \wedge \Phi_1^\ast(\eta)} = \int_{\B^{4n}} \Phi_1 (\eta \wedge \eta).
\]
With the rank-condition, Lemma~\ref{la:rank}, since $\eta \wedge \eta$ is a $4n$-form,
\[
 0 \neq  \HI (\varphi_1) =0.
\]
We have our contradiction.

For the density argument we argue as above. Since we have a quantification of the nontriviality, $0 \neq  \HI (\varphi_1)$, we simply need to check convergence for Sobolev spaces.

\section{Approximation and rank conditions for H\"older-maps}
Essentially all the above arguments crucially rely on the rank-condition, that any Lipschitz map $\varphi \in \lip(\Omega,\H_n)$ has $\rank D\varphi \leq n$. 
For H\"older maps $\varphi$, there is no derivative $D\varphi$ which could have a rank. So we approximate H\"older maps $\varphi \in C^\nu(\Omega,\H_n)$ with smooth maps $\varphi_\eps \in C^\infty(\Omega,\R^{2n+1})$ (e.g. by mollification). Note, however, that there is absolutely no reason why the approximations $\varphi_\eps$ are Lipschitz maps as maps into the Heisenberg group. 

The main observation to overcome this issue is the following:

\begin{proposition}\label{pr:approx}
For $\varphi \in C^\gamma(\Omega,\H_n)$ there exist $\varphi_\eps \in C^\infty(\Omega,\R^{2n+1})$ so that $\varphi_\eps \to f$ in $C^\gamma(\Omega,\R^{2n+1})$ and moreover
\begin{equation}\label{eq:firstorderest}
 \|\varphi_\eps^\ast (\alpha) \|_\infty \aleq \eps^{2\gamma-1},
\end{equation}
where $\alpha$ is the contact form \eqref{eq:alpha}, and
\begin{equation}\label{eq:approxgeneral}
 \|\varphi_\eps^\ast (\kappa) \|_\infty \aleq \eps^{k(\gamma-1)},
\end{equation}
for any $k$-form $\kappa$.
\end{proposition}
While \eqref{eq:approxgeneral} is the standard estimate for approximations ($\|D \varphi_\eps\| \aleq \eps^{\gamma-1} [\varphi]_{C^\gamma}$), \eqref{eq:firstorderest} gives us ``a special direction'' in which the approximation is better - if $\gamma > \frac{1}{2}$ it is actually convergent.

In particular, we have the following replacement for $\rank D\phi \leq n+1$:
\begin{proposition}\label{pr:hoeldermapsintoHn}
Let $\Phi \in C^{\nu}(\overline{\B^{k+1}},\H_n)$ with boundary data $\varphi = \Phi \big|_{\S^k}$. If $\nu > \frac{k+1}{k+2}$ and $\kappa$ is any smooth $k$-form on $\R^{2n+1}$, $k \geq n$, for $\Phi_\eps$ the approximation of $\Phi$ as in Proposition~\ref{pr:approx},
\[
 \lim_{\eps \to 0} \int_{\S^{k}} \varphi_\eps^\ast(\kappa) = 0.
\]
\end{proposition}
\begin{proof}
We have with Stokes' theorem
\[
 \int_{\S^{k}} \varphi_\eps^\ast(\kappa) = \int_{\B^{k+1}} \Phi_\eps^\ast(d\kappa).
\]
In view of Lemma~\ref{la:rank}, $d\kappa = \alpha \wedge \beta + d\alpha \wedge \sigma$
\[
 = \int_{\B^{k+1}} \Phi_\eps^\ast(\alpha)\wedge \Phi_\eps^\ast(\beta)+ \int_{\B^{k+1}} \Phi_\eps^\ast(d\alpha)\wedge \Phi_\eps^\ast(\sigma)
\]
and again Stokes' theorem
\[
 = \int_{\B^{k+1}} \Phi_\eps^\ast(\alpha)\wedge \Phi_\eps^\ast(\beta) + \int_{\S^{k}} \Phi_\eps^\ast(\alpha)\wedge \Phi_\eps^\ast(\sigma) - \int_{\B^{k+1}} \Phi_\eps^\ast(\alpha)\wedge d\Phi_\eps^\ast(\sigma).
\]
Now with \eqref{eq:firstorderest} and \eqref{eq:approxgeneral},
\[
 \left |\int_{\S^{k}} \varphi_\eps^\ast(\kappa)\right | \aleq \eps^{2\gamma-1} \eps^{k(\gamma-1)} \xrightarrow{\eps \to 0} 0,
\]
whenever $\nu > \frac{k+1}{k+2}$.

\end{proof}

\section{The linking number: Proof of Proposition~\ref{pr:linking}}\label{s:linkingproof}
For H\"older maps we need to adapt Proposition~\ref{pr:linking}.

Let $k < N-1$ and $\varphi: \S^k \to \R^N$ be a $C^{\sigma}$-embedding for $\sigma > \frac{k}{k+1}$. For an exact form $\eta = d\omega \in C^\infty(\Ep^{k+1} \R^N)$, we define the linking number between $\eta$ and $\varphi(\S^k)$ by
\begin{equation}\label{eq:linking}
 \mathcal{L}(\varphi(\S^k),\eta) := \lim_{\eps \to 0} \int_{\S^k} \varphi_\eps^\ast(\omega).
\end{equation}
Here, $\varphi_\eps$ is any smooth approximation of $\varphi$ in $C^\sigma$.

\begin{lemma}\label{la:lconverges}
If $\sigma > \frac{k}{k+1}$, \eqref{eq:linking} converges and is independent of the choice of the approximation.
\end{lemma}
\begin{proof}
This can be proven in various ways: with the help of paraproducts \cite{Sickel-Youssfi-1999a}, Fourier transform estimates a la \cite{Tartar84}. Actually this convergence is in some sense related to ``integration by compensation'' for Jacobians, as observed by Coifman-Lions-Meyer-Semmes \cite{CLMS}, see also \cite{LenzmanN-Schikorra-commutators}). 

The simplest argument (for our purposes) is a beautiful trick, due to Brezis and Nguyen \cite{Brezis-NguyeN-2011}. Take $\Phi_\eps$ the harmonic extension of $\varphi_\eps$ in $\B^{k+1}$, $\lap \Phi_\eps = 0$, $\Phi_\eps \big |_{\S^{k}} = \varphi_\eps$. Then, by Stokes theorem,
\[
\int_{\S^k} \varphi_\eps^\ast(\omega) = \int_{\B^{k+1}} \Phi_\eps^\ast(d\omega).
\]
Now observe that $d\omega$ is a (bounded) $k+1$-form, so
\[
\left |\int_{\B^{k+1}} \Phi_\eps^\ast(d\omega) \right |\aleq \|D \Phi_\eps\|_{L^{k+1}(\B^{k+1})}^{k+1}.
\]
But $\Phi_\eps$ is an extension of $\varphi_\eps$, in other words, $\varphi_\eps$ is the trace of the harmonic function $\Phi_\eps$. Trace theorems for Sobolev mappings $W^{1,k+1}(\B^{k+1}) \hookrightarrow W^{\frac{k}{k+1},k+1}(\partial \B^{k+1})$ imply
\[
 \left |\int_{\B^{k+1}} \Phi_\eps^\ast(d\omega) \right |\aleq \| \varphi_\eps\|^{k+1}_{W^{\frac{k}{k+1},k+1}(\S^k)} \aleq \|\varphi_\eps \|_{C^{\sigma}}^{k+1}.
\]
Using this argument one can show that 
\[
 \int_{\S^k} \varphi_\eps^\ast(\omega) 
\]
is a Cauchy sequence as $\eps \to 0$, in particular, \eqref{eq:linking} is converging.
\end{proof}
Observe that for maps as in Proposition~\ref{pr:hoeldermapsintoHn} we thus have that the linking number is necessarily zero, which then -- just as in Section~\ref{s:lipschitzgromov} contradicts the following proposition, which is simply the extension to H\"older maps from Proposition~\ref{pr:linking}.

\begin{proposition}\label{pr:linkinghoelder}
Let $k < N-1$ and $\varphi: \S^k \to \R^N$ a $C^{\sigma}$-embedding for $\sigma > \frac{k}{k+1}$. Then there exists a smooth $k$-form $\omega$ on $\R^N$ so that
\[
 \mathcal{L}(\varphi(\S^k),d\omega) \neq 0.
\]
\end{proposition}

\begin{remark}
So we can \emph{measure} the linking number for $C^{\frac{k}{k+1}+\eps}$-embeddings $\varphi: \S^k \to \R^{2n+1}$. By standard algebraic arguments (essentially the arguments we do below), for all $C^\sigma$-embedding, even if $\sigma \leq \frac{k}{k+1}$ there is a linked object $B$ so that the \emph{algebraic} linking number is nontrivial -- which is just saying that the cohomology group $H_c^{k+1}(\R^{2n+1}\backslash \varphi(\S^k)) \neq 0$. But this algebraic linking number we cannot ``measure'' in analytic terms.

Moreover, note that in Proposition~\ref{pr:hoeldermapsintoHn} we can only show for $C^{\frac{k+1}{k+2}+\eps}$-embeddings $\varphi$ into the Heisenberg group that our \emph{analytic} linking number is always zero.

So only when $\varphi$ is a  $C^{\frac{k+1}{k+2}+\eps}$-embedding into the Heisenberg group can we compare the algebraic linking number (nonzero, since it is an embedding) and the analytic linking number (zero, since it is a $C^{\frac{k+1}{k+2}+\eps}$-map into the Heisenberg group).
\end{remark}

\subsection{Proof of Proposition~\ref{pr:linkinghoelder}}
We split a sphere $\S^\ell$ into its equator, which we denote $\S^{\ell-1}$ and its closed upper hemisphere $S^\ell_+$ and lower hemisphere $\S^\ell_-$, i.e. $\S^{\ell}_+ \cap \S^{\ell}_- = \S^{\ell-1}$. We argue by induction on the dimension of the sphere $\S^\ell$, $\ell = 0,\ldots,k$.

We will pretend that t$\varphi$ is a Lipschitz map purely for notational reasons. The argument works exactly as is for the $C^\sigma$-embeddings, everything is just a matter of supports.

The induction claim is
\[
 \tag{I} \forall \ell = 0,\ldots, k: \quad \exists \mbox{ $\omega_\ell$, a smooth $\ell$-form, $d\omega_\ell = 0$ around $\varphi(\S^\ell)$, and $\int_{\S^\ell} \varphi^\ast(\omega_\ell) \neq 0$}
\]

\subsubsection*{Case $\ell = 0$}
By the decomposition above, $\S^0$ are simply to points, which we may denote with $\{-1,+1\}$. Since $\varphi$ is an embedding, $\varphi(-1) \neq \varphi(+1)$. So we just pick $\omega_0$ a $0$-form (i.e. function on $\R^N$) to be constantly $1$ around $\varphi(-1)$ and constantly $-1$ around $\varphi(+1)$. Then $d\omega_0 = 0$ around $\varphi(\S^0)$, and
\[
 \int_{\S^0} \varphi^\ast(\omega_0) = \omega_0(\varphi(1))-\omega_0(\varphi(-1)) = 2\neq 0
\]

\subsubsection*{Case $(\ell-1) \to \ell$}
We assume that we have found an $\ell$-form $\omega_{\ell-1}$, $\eta_{\ell-1} := d\omega_{\ell-1}$ is zero around $\varphi(\S^{\ell-1})$, and
\[
 \int_{\S^{\ell-1}} \varphi^\ast(\omega_{\ell-1}) \neq 0.
\]
Having $\eta_{\ell-1} = d\omega_{\ell-1}$ we first construct a closed $(\ell+1)$-form $\eta_{\ell}$.
%

Define open subsets of $\R^N$ as follows: $U: = \R^N \backslash \varphi(\S^\ell_+)$, $V: = \R^N \backslash \varphi(\S^\ell_-)$.

The support of $\eta_{\ell-1}$ is bounded away from $\varphi(\S^{\ell-1})$, thus \[\supp \eta_{\ell-1} \subset \R^N \backslash \varphi(\S^{\ell-1}) = U \cup V\]
By a cutoff-argument, since , we can find two $\ell$-forms $\gamma_{U}$ and $\gamma_{V}$ supported in $U$ and $V$, respectively, and so that
\begin{equation}\label{eq:etalm1split}
 \eta_{\ell-1} = \gamma_U + \gamma_V
\end{equation}
We define
\[
 \omega_{\ell} := \gamma_U, \quad \eta_{\ell} := d\gamma_U.
\]
Since $d\eta_{\ell-1}=d\circ d\omega_{\ell-1} = 0$ we actually have 
\[
\eta_{\ell} =d\gamma_U = -d\gamma_V. 
\]
In particular,
\[
 \supp \eta_\ell \subset \supp \gamma_U  \cap \supp \gamma_V \subset U \cap V \subset \R^N \backslash \varphi(\S^\ell).
\]
Thus, we have found
\[
\omega_\ell,
 \mbox{ smooth $\ell$-form, $d\omega_\ell \equiv \eta_\ell = 0$ around $\varphi(\S^\ell)$},
\] i.e., $\omega_\ell$ is almost as needed for the induction claim $(I)$, we just need to confirm that
\begin{equation}\label{eq:stepintnzero}
 \int_{\S^{\ell}} \varphi^\ast(\omega_{\ell}) \neq 0.
\end{equation}
So let us compute \eqref{eq:stepintnzero}. In view of the support of $\gamma_U$ and $\gamma_V$ and \eqref{eq:etalm1split}
\[
 \int_{\S^{\ell}} \varphi^\ast(\omega_{\ell}) = \int_{\S^{\ell}_-} \varphi^\ast(\gamma_U) = \int_{\S^{\ell}_-} \varphi^\ast(\eta_{\ell-1}-\gamma_V) = \int_{\S^{\ell}_-} \varphi^\ast(\eta_{\ell-1}).
\]
Now we use Stokes' theorem on $\S^\ell_-$. Observe that by the orientation of $\partial \S^\ell_- = -\S^{\ell-1}$ we get a sign.
\[
 \int_{\S^{\ell}_-} \varphi^\ast(\eta_{\ell-1}) = \int_{\S^{\ell}_-} \varphi^\ast(d\omega_{\ell-1}) =  \int_{\S^{\ell-1}} \varphi^\ast(\omega_{\ell}).
\]
That is, we have by induction hypothesis
\[
 \int_{\S^{\ell}} \varphi^\ast(\omega_{\ell})  = -\int_{\S^{\ell-1}} \varphi^\ast(\omega_{\ell-1}) \neq 0,
\]
and \eqref{eq:stepintnzero} is proven.

\begin{remark}
Let us put the above argument into perspective of algebraic topology. By induction hypothesis, $\eta_{\ell-1}$ is an element of the cohomology group $H^{\ell}(\R^N \backslash \varphi(\S^{\ell-1}))$. We just used the exact Mayer-Vietoris sequence,
\[
\ldots \to H^{\ell}(U) \oplus H^{\ell}(U) \to  H^{\ell}(U \cup V) \xrightarrow{c} H^{\ell+1}(U \cap V) \to H^{\ell+1}(U) \oplus H^{\ell+1}(U) \to \ldots 
\]
where we observe $U \cap V = \R^N \backslash \varphi(\S^{\ell})$, $U \cup V = \R^N \backslash \varphi(\S^{\ell-1})$. Also, since $U$ and $V$ are homeomorphic to $\R^N$ (that is $\S^{N}$) with a cube taken away, \[H^{\ell+1}(U) = H^{\ell+1}(V) = H^{\ell}(U) =H^{\ell}(V) = 0.\]
Thus, the Mayer-Vietoris sequence is simply
\[
0 \to  H^{\ell}( \R^N \backslash \varphi(\S^{\ell-1})) \xrightarrow{c} H^{\ell+1}(\R^N \backslash \varphi(\S^{\ell})) \to 0.
\]
This just means that the connecting homomorphism $c: H^{\ell}(\R^N \backslash \varphi(\S^{\ell-1})) \xrightarrow{c} H^{\ell+1}(\R^N \backslash \varphi(\S^{\ell})))$ is an isomorphism. On the other hand $c$ is known, and all we did above is set $\eta_{\ell} := c(\eta_{\ell-1})$.

Actually one can show that $\eta = d\omega \mapsto \int_{\S^k} \varphi^\ast(\omega)$ is an isomorphism on $H^{k+1}(\R^N \backslash \varphi(\S^k))$.
\end{remark}

\section*{Acknowledgement}
This text was mainly written while the author was preparing a lecture for the 19th Rencontres d'Analyse at UCLouvain in October 2016. He likes express his gratitude to UCLouvain and the organizers Pierre Bousquet, Jean Van Schaftingen, Augusto Ponce for the kind invitation and their hospitality.

\bibliographystyle{abbrv}
\bibliography{bib}%

\end{document}